\documentclass{amsart}
\usepackage[all]{xy}
\SelectTips{cm}{10}
\usepackage{color}
\usepackage{amsthm}
\usepackage{amssymb}
\usepackage{stmaryrd}
\usepackage{url}
\usepackage[colorlinks=true]{hyperref}

\numberwithin{equation}{section}

\newtheorem{thm}[equation]{Theorem}
\newtheorem*{thm*}{Theorem}
\newtheorem{lem}[equation]{Lemma}

\newtheorem*{conj*}{Conjecture}
\newtheorem{prop}[equation]{Proposition}
\newtheorem{cor}[equation]{Corollary}
\newtheorem*{corollary*}{Corollary}

\theoremstyle{definition}
\newtheorem{defn}[equation]{Definition}

\newtheorem{rem}[equation]{Remark}

\begin{document}

\title{$K'$-Theory of a Local Ring of Finite Cohen-Macaulay Type}
\author{Viraj Navkal}
\thanks{This work was partially supported by the National Science Foundation Award DMS-0966821.}
\address{Viraj Navkal, Mathematics Department, UCLA, Los Angeles, CA 90095, USA}
\email{viraj@math.ucla.edu}
\date{}

\begin{abstract}
We study the $K'$-theory of a CM Henselian local ring $R$ of finite Cohen-Macaulay type.  We first describe a long exact sequence involving the groups $K_i'(R)$ and the $K$-groups of certain other rings, including the Auslander algebra.  By examining the terms and maps in the sequence, we obtain information about $K'(R)$.
\end{abstract}

\maketitle
\vskip-\baselineskip
\vskip-\baselineskip
\vskip-\baselineskip

\section{Introduction}
\label{introduction}
This paper generalizes the following result of Auslander and Reiten; see \cite[$\S 2$, Prop. 2.2]{AusRei86} for the original statement or \cite[13.7]{Yos90} for the statement in this form.
\begin{thm} \label{AR}
Let $R$ be a Henselian CM local ring of finite representation type.  Denote by $H$ the free abelian group on the set of isomorphism classes of indecomposable maximal Cohen-Macaulay $R$-modules.  Then the map $H \longrightarrow K_0'(R)$ sending $[M]$ to $[M]$ is surjective, and its kernel is the subgroup
\[
\langle [N]-[E]+[M]\ |\ \exists \mathrm{\emph{ an Auslander-Reiten sequence}}\ 0 \longrightarrow N \longrightarrow E \longrightarrow M \longrightarrow 0 \rangle .
\]
\end{thm}
\vspace{1.5 mm}
Thus if one knows the Auslander-Reiten sequences in the category of maximal Cohen-Macaulay $R$-modules, the above theorem allows explicit computation of $K_0'(R)$.

Using Auslander and Reiten's techniques, we are able to produce a long exact sequence involving the groups $K_i'(R)$ (see (\ref{eq:3})).
\begin{thm}
\label{sequencetheorem}
Let $R$ be a Henselian CM local ring of finite Cohen-Macaulay type.  Let $I$ be the set of isomorphism classes of indecomposable maximal Cohen-Macaulay $R$-modules, and let $I_0  = I\setminus \{[R]\}$.  Then there is a long exact sequence
\begin{equation}
\label{sequence}
\cdots \rightarrow \bigoplus \limits_{[M] \in I_0} K_i(\kappa _M) \rightarrow K'_i(\Lambda) \rightarrow K_i'(R) \rightarrow \bigoplus \limits_{[M] \in I_0} K_{i-1}(\kappa _M) \rightarrow \cdots
\end{equation}
where $\Lambda = \mathrm{End}_R (\bigoplus \limits_{[M] \in I} M)^{\mathrm{op}}$ and $\kappa _M = (\mathrm{End}_R M)^{\mathrm{op}} / \mathrm{rad}((\mathrm{End}_R M)^{\mathrm{op}})$.
The long exact sequence ends in a presentation
\[ \xymatrix @R=0.9pc{
\bigoplus \limits_{[M] \in I_0}K_i(\kappa _M) \ar[r]	& K'_0(\Lambda) \ar[r]		& K_0'(R) \ar[r]	& 0 \\
\mathbb{Z}^{I_0} \ar@{=}[u]				& \mathbb{Z}^I \ar@{=}[u]	& 			&
} \]
of $K_0'(R)$; it is exactly the one described in Theorem \ref{AR}.
\end{thm}

In the case when $R = S/(w)$ is a hypersurface singularity of finite CM type, we derive a similar long exact sequence involving the algebraic $K$-groups of the category $\mathsf{MF}$ of matrix factorizations with potential $w$; see Remark \ref{mf}.

In section \ref{ausl-reit-matr}, we study the maps $\alpha _i:\bigoplus K_i(\kappa _M) \rightarrow K_i'(\Lambda)$ appearing in (\ref{sequence}).  Not surprisingly, we find a relationship between these higher maps and the matrix $T$ defining the map $\alpha _0:\mathbb{Z}^{I_0} = \bigoplus K_0(\kappa _M) \rightarrow K_0'(\Lambda) = \mathbb{Z}^{I}$.  This matrix $T$ was first studied in \cite{AusRei86}.

In Section \ref{an-example}, we apply the results of the previous sections to understand $K_1'(R)$ and $K_1(\mathsf{MF})$ explicitly when $R$ is a one-dimensional singularity of type $A_{2n}$.  For this computation, we need to know more about the term $K_1'(\Lambda)$ in the sequence; to this end we study $K_1$ of a Krull-Schmidt additive category.  Our work in this direction appears in Appendix \ref{loca-for-a-krul-schm-cate}.

I would like to thank my advisor Christian Haesemeyer for his guidance and support; without him this work would not have been possible.  I am grateful to the referee for catching a serious error in the original version of this paper.  Also I thank Daniel Grayson for his assistance.

\section{Preliminaries}
\label{preliminaries}

All categories in this paper happen to be additive, and all functors between them are assumed to be additive.  All modules over a ring are left modules.  All spectra are viewed as objects in the stable homotopy category $\mathrm{Ho}(\mathsf{Sp})$, and maps between spectra are viewed as morphisms in this category.  This is useful because $\mathrm{Ho}(\mathsf{Sp})$ is a triangulated category, and the stable homotopy functor $\mathrm{Ho}(\mathsf{Sp}) \rightarrow (\mathrm{abelian\ groups})$ is a homological functor.

If $\mathcal{A}$ is an additive category and $M_1,\ldots ,M_n \in \mathcal{A}$, $\mathsf{add}(M_1,\ldots ,M_n)$ will denote the full subcategory of $\mathcal{A}$ consisting of direct summands of direct sums of $M_1,\ldots ,M_n$; in particular $\mathsf{add} (M_1 \oplus \cdots \oplus M_n) = \mathsf{add}(M_1,\ldots ,M_n)$.  $\mathrm{rad}(\mathcal{A})$ will denote the Jacobson radical of $\mathcal{A}$; this is a two-sided ideal of $\mathcal{A}$ and can be defined by
\[
\mathrm{Hom}_{\mathrm{rad}(\mathcal{A})}(X,Y)=\{f \in \mathrm{Hom}_{\mathcal{A}}(X,Y)|fg \in \mathrm{rad}(\mathrm{End}_{\mathcal{A}}Y)\mathrm{\ for\ all\ }g \in \mathrm{Hom}_{\mathcal{A}}(Y,X)\}.
\]

Given a collection of additive categories $\{\mathcal{A}_i\}_{i \in I}$, we define their direct sum $\bigoplus \mathcal{A}_i$ to be the additive category whose objects are $I$-tuples $X=(X_i)_{i \in I}$ with  $X_i$ in $\mathcal{A}_i$ for all $i$ and $X_i \neq 0$ for only finitely many $i$.  Morphisms are defined by $\mathrm{Hom}_{\bigoplus \mathcal{A}_i}(X,Y)=\bigoplus \mathrm{Hom}_{\mathcal{A}_i}(X_i,Y_i)$.

\begin{defn}
A nonzero object $A$ of an additive category $\mathcal{A}$ is called \emph{indecomposable} if whenever $A \cong A' \oplus A''$, $A' \cong 0$ or $A'' \cong 0$.  Denote by $\mbox{ind}(\mathcal{A})$ the set of isomorphism classes of indecomposable objects in $\mathcal{A}$.
\end{defn}
\begin{defn}
We call the additive category $\mathcal{A}$ \emph{Krull-Schmidt} if any object of $\mathcal{A}$ is a finite direct sum of objects with local endomorphism rings.
\end{defn}
By the Krull-Schmidt theorem, every object in a Krull-Schmidt category can be written uniquely as a finite direct sum of indecomposables \cite[3.2.1]{Kra12}.

\begin{defn}
An \emph{exact category} is an additive category $\mathcal{E}$ together with a distinguished class of sequences $\xymatrix{M'\  \ar@{>->}[r] & M \ar@{>>}[r] & M''}$, called \emph{conflations}, such that there is a fully faithful additive functor $f$ from $\mathcal{E}$ into an abelian category $\mathcal{A}$ satisfying the following two properties:
\begin{itemize}
\item[1.]
$f$ reflects exactness; that is, $\xymatrix{M'\  \ar@{>->}[r] & M \ar@{>>}[r] & M''}$ is a conflation in $\mathcal{E}$ if and only if $0 \rightarrow f(M') \rightarrow f(M) \rightarrow f(M'') \rightarrow 0$ is a short exact sequence in $\mathcal{A}$.
\item[2.]
$\mathcal{E}$ is closed under extensions in $\mathcal{A}$; that is, if $0 \rightarrow f(M') \rightarrow A \rightarrow f(M'') \rightarrow 0$ is exact in $\mathcal{A}$, then $A \cong f(M)$ for some $M$ in $\mathcal{E}$.
\end{itemize}
We say $\mathcal{E}'$ is an \emph{exact subcategory} of $\mathcal{E}$ if $\mathcal{E}'$ and $\mathcal{E}$ are exact categories, $\mathcal{E}'$ is a subcategory of $\mathcal{E}$, the inclusion functor $\mathcal{E}' \rightarrow \mathcal{E}$ reflects exactness, and $\mathcal{E}'$ is closed under extensions in $\mathcal{E}$.
\end{defn}
If each $\mathcal{E}_i$ is an exact category, we endow $\bigoplus \mathcal{E}_i$ with the structure of exact category by setting the conflations to be the sequences which are coordinate-wise conflations.  This direct sum is then the coproduct with respect to exact functors between exact categories.

If $R$ is a ring, the category $\mathsf{mod}(R)$ of finitely presented $R$-modules is an exact category in which the conflations are the short exact sequences.  Similarly, the category $\mathsf{proj}(R)$ of finitely generated projective $R$-modules is an exact category with short exact sequences for conflations.  Note that every conflation in $\mathsf{proj}(R)$ is split.

A subcategory $w\mathcal{E}\subset \mathcal{E}$ of an exact category $\mathcal{E}$ is called a \emph{subcategory of weak equivalences} if it contains all objects of $\mathcal{E}$ and all isomorphisms in $\mathcal{E}$ and satisfies Waldhausen's Gluing Lemma \cite[1.2]{Wal83}.  We define the $K$-theory of an exact category $\mathcal{E}$ relative to a subcategory of weak equivalences $w\mathcal{E} \subset \mathcal{E}$ using Waldhausen's $S_{\cdot}$-construction \cite[1.3]{Wal83}.  It is denoted simply by $K(\mathcal{E})$, and it is an $\Omega$-spectrum whose $n$th space is
\[
K(\mathcal{E})_n = \Omega |wS_{\cdot}^n\mathcal{E}|.
\]
If no subcategory of weak equivalences is specified, the $K$-theory of an exact category $\mathcal{E}$ is taken relative to the subcategory $i\mathcal{E}$ of isomorphisms.

Set $K_n(\mathcal{E}) = \pi _n (K(\mathcal{E}))$.  For a ring $R$, set
\begin{align*}
K(R) & := K(\mathsf{proj}(R))	& K_i(R) & := \pi _i(K(R)) \\
K'(R) & := K(\mathsf{mod}(R))	& K_i'(R) & := \pi _i(K'(R))
\end{align*}
Then $K_0(R)$ is the usual Grothendieck group of $R$.  If $R$ is a commutative local ring, $K_1(R) \cong R^{\times}$.
\section{The Long Exact Sequence}
\label{long-exact-sequence}
Fix a Henselian Cohen-Macaulay local ring $R$ with maximal ideal $\mathfrak{m}$.  Assume also that $R$ has a canonical module, and that $R$ is an isolated singularity, in the sense that $R_{\mathfrak{p}}$ is regular for nonmaximal primes $\mathfrak{p}$.  As $R$ is Henselian local, $\mathsf{mod}(R)$ is a Krull-Schmidt category (\cite[Lemma 13]{Vam90}, \cite{Sid90}).

\begin{defn}
\label{sec:main-result-1}
A finitely generated $R$-module $M$ is called \emph{maximal Cohen-Macaulay} if it satisfies the following two equivalent conditions:
\begin{itemize}
\item[1.]
$\mbox{depth } M = \mbox{dim } R$
\item[2.]
$\mathrm{Ext}_R^i(R/\mathfrak{m},M) = 0$ for $i < \mathrm{dim}(R)$
\end{itemize}
Denote by $\mathcal{C}$ the exact subcategory of $\mathsf{mod}(R)$ consisting of the maximal Cohen-Macaulay modules.  Let $\mathcal{C}^{\oplus}$ be the same category as $\mathcal{C}$ but with a different exact structure: the conflations in $\mathcal{C}^{\oplus}$ are the split short exact sequences.
\end{defn}

Since $\mathcal{C}$ is closed under summands and $\mathsf{mod}(R)$ is Krull-Schmidt, $\mathcal{C}$ also is Krull-Schmidt.  In addition, using \ref{sec:main-result-1}.2, one can easily prove that $\mathcal{C}$ satisfies the following.

\begin{lem}
\label{sec:main-result-2}
Suppose
\[ \xymatrix{
0 \ar[r]	& M' \ar[r]	& M \ar[r]	& M'' \ar[r]	& 0
} \]
is a short exact sequence of $R$-modules.
\begin{itemize}
\item[1.]
If $M'$ and $M''$ are in $\mathcal{C}$, then so is $M$.
\item[2.]
If $M$ and $M''$ are in $\mathcal{C}$, then so is $M'$.
\end{itemize}
\end{lem}

\begin{defn}
Let $\widetilde{\mathcal{C}}$ be the abelian category of contravariant additive functors from $\mathcal{C}$ to the category of abelian groups.  Let $\widehat{\mathcal{C}}$ be the full subcategory of $\widetilde{\mathcal{C}}$ consisting of functors $F$ which fit into an exact sequence
\[ \xymatrix{
\mathrm{Hom}_{\mathcal{C}}(-,M) \ar[r]	& \mathrm{Hom}_{\mathcal{C}}(-,M') \ar[r]	& F \ar[r]	& 0
} \]
with $M$ and $M'$ in $\mathcal{C}$; such functors are called \emph{finitely presented}.  $\widehat{\mathcal{C}}$ is an abelian category \cite[4.19]{Yos90}; kernels and cokernels are computed pointwise, as in $\widetilde{\mathcal{C}}$.
\end{defn}

For the next proposition we will need Quillen's resolution theorem.  For convenience we reproduce it here, in a form that is useful for us.
\begin{thm}[{\cite[Cor. 1 to Thm. 3]{Qui73}}]
\label{resolution}
Let $\mathcal{E}$ be an exact category and $\mathcal{P} \subset \mathcal{E}$ an exact subcategory.  Suppose that
\begin{itemize}
\item[1.]
for any conflation $\xymatrix{M\  \ar@{>->}[r] & P \ar@{>>}[r] & P'}$ with $P$ and $P'$ in $\mathcal{P}$, $M$ is isomorphic to an object in $\mathcal{P}$; and
\item[2.]
every object of $\mathcal{E}$ has a $\mathcal{P}$-resolution of finite length.
\end{itemize}
Then the inclusion functor induces a homotopy equivalence $K(\mathcal{P}) \simeq K(\mathcal{E})$.
\end{thm}

\begin{prop}
\label{prop1}
The Yoneda functor $h:\mathcal{C}^{\oplus} \rightarrow \widehat{\mathcal{C}}$ is a $K$-theory equivalence.
\end{prop}

\begin{proof}
We verify the hypotheses of the resolution theorem.  We need to show that the essential image of $h$ is closed under extensions.  Suppose
\[ \xymatrix{
0 \ar[r]	& h(A) \ar[r]	& F \ar[r]^-p	& h(B) \ar[r]	& 0
} \]
is exact in $\widehat{\mathcal{C}}$, and choose a lift $x \in F(B)$ of $\mathrm{id}_B \in h(B)(B)$.  Define $i:h(B) = \mathrm{Hom}_{\mathcal{C}}(-,B) \rightarrow F$ by $i_C(f) = F(f)(x)$, for $f:C \rightarrow B$.  Then $i$ splits $p$, so $F$ is isomorphic to $h(A) \oplus h(B) \cong h(A \oplus B)$.

We need also to show that if
\[ \xymatrix{
0 \ar[r]	& F \ar[r]	& h(A) \ar[r]^p	& h(B) \ar[r]	& 0
} \]
is exact in $\widehat{\mathcal{C}}$, then $F$ is in the essential image of $h$.  As before, one can find a section of $p$, which must be of the form $h(i)$ for some split monomorphism $i:B \rightarrow A$.  Then $F \cong h(\mathrm{coker}(i))$.

Last, we must show that every functor in $\widehat{\mathcal{C}}$ has a finite resolution by representable functors.  By Auslander-Buchweitz approximation \cite{AusBuc89}, for every $X \in \mathsf{mod}(R)$ there is a map $f:M \rightarrow X$ with $M$ in $\mathcal{C}$ such that $\mathrm{Hom}(-,f):\mathrm{Hom}(-,M) \rightarrow \mathrm{Hom}(-,X)$ is an epimorphism between functors in $\widetilde{\mathcal{C}}$.  By repeatedly taking approximations to syzygies, we find that any $F$ in $\widehat{\mathcal{C}}$ has a resolution
\[ \xymatrix{
\cdots \ar[r]^{h(f_2)}	& h(M_1) \ar[r]^{h(f_1)}	& h(M_0) \ar[r]	& F \ar[r]	& 0
} \]
with $M_i$ in $\mathcal{C}$ for $i \geq 0$.  Evaluating this sequence at $R$ produces an exact sequence of $R$-modules
\[ \xymatrix{
\cdots \ar[r]^{f_2}	& M_1 \ar[r]^{f_1}	& M_0 \ar[r]^{f_0}	& F(R) \ar[r]	& 0.
} \]
By the exact sequences
\[ \xymatrix{
0 \ar[r]	& \mathrm{im\,} f_{n+1} \ar[r]	& M_n \ar[r]	& \mathrm{im\,} f_n \ar[r]	& 0
} \]
it follows that whenever $n \geq \mathrm{dim\,} R$, $\mathrm{Ext}^i_R(R/\mathfrak{m},\mathrm{im}\,f_n) = 0$ for $i < \mathrm{dim\,}R$.  Therefore $\mathrm{im}\,f_{\mathrm{dim}(R)}$ is maximal Cohen-Macaulay and
\[ \xymatrix{
0 \ar[r]	& h(\mathrm{im}\,f_{\mathrm{dim}(R)}) \ar[r]	& h(M_{\mathrm{dim}(R)-1}) \ar[r]	& \cdots \ar[r]	& h(M_0) \ar[r]	& F \ar[r]	& 0
} \]
is a resolution of $F$ by representable functors.
\end{proof}

Recall (\cite[I.2]{Swa68}) that if $\mathcal{A}$ is a well-powered abelian category and $\mathcal{B} \subset \mathcal{A}$ a Serre subcategory, there is an abelian category $\mathcal{A}//\mathcal{B}$ (usually denoted $\mathcal{A}/\mathcal{B}$, but we reserve this notation for the additive quotient), together with an exact functor $\mathcal{A} \rightarrow \mathcal{A}//\mathcal{B}$ which is universal for exact functors out of $\mathcal{A}$ which vanish on $\mathcal{B}$.  The following criterion for recognizing the quotient category is from \cite[I.1.6]{AusRei86}; we omit the proof.

\begin{lem}
\label{auslandercriterion}
Let $e:\mathcal{A} \rightarrow \mathcal{D}$ be an exact functor between well-powered abelian categories, and let $p:\mathcal{A} \rightarrow \mathcal{A}//\mathrm{ker}\,e$ be the canonical functor.  Suppose there is an additive functor $s:\mathcal{D} \rightarrow \mathcal{A}$ such that $es = \mathrm{id}_{\mathcal{D}}$.  Then $ps:\mathcal{D} \rightarrow \mathcal{A}//\mathrm{ker}\,e$ is an exact functor and an inverse isomorphism to the functor $\overline{e}:\mathcal{A}//\mathrm{ker}\,e \rightarrow \mathcal{D}$ induced by $e$.
\end{lem}

For the next theorem we will need Quillen's localization theorem; we reproduce it here.
\begin{thm}[{\cite[Theorem 5]{Qui73}}]
\label{localization}
Let $\mathcal{A}$ be an abelian category with a set of isomorphism classes of objects.  Let $\mathcal{B} \subset \mathcal{A}$ be a Serre subcategory, $i:\mathcal{B} \rightarrow \mathcal{A}$ the inclusion functor, and $p:\mathcal{A} \rightarrow \mathcal{A}//\mathcal{B}$ the canonical functor.  Then
\[ \xymatrix{
K(\mathcal{B}) \ar[r]^{K(i)}	& K(\mathcal{A}) \ar[r]^-{K(p)}	& K(\mathcal{A}//\mathcal{B})
} \]
is a homotopy fiber sequence.
\end{thm}

Let $e:\widehat{\mathcal{C}} \rightarrow \mathsf{mod}(R)$ be the evaluation functor $F \mapsto F(R)$.  Let $\widehat{\mathcal{C}}_0=\mathrm{ker}(e) \subset \widehat{\mathcal{C}}$, i.e. $\widehat{\mathcal{C}}_0$ is the category of finitely presented functors $F$ satisfying $F(R) \cong 0$.  $\widehat{\mathcal{C}}_0$ is an abelian category (\cite[4.17]{Yos90}); kernels and cokernels in $\widehat{\mathcal{C}}_0$ are computed pointwise.
\begin{thm}[{\cite[I.1.5]{AusRei86}}]
Let $i:\widehat{\mathcal{C}}_0 \rightarrow \widehat{\mathcal{C}}$ be the inclusion functor.  Then
\[
\xymatrix{
K(\widehat{\mathcal{C}}_0) \ar[r]^{K(i)}	& K(\widehat{\mathcal{C}}) \ar[r]^{K(e)}	& K'(R)
}
\]
is a homotopy fiber sequence.
\label{thm2}
\end{thm}

\begin{proof}
Since $\widehat{\mathcal{C}}_0 = \mathrm{ker}\,e$, it suffices by Lemma \ref{auslandercriterion} and the localization theorem \ref{localization} to find an additive functor $s$ which is right inverse to $e$.  For each $M$ in $\mathsf{mod}(R)$, choose a projective presentation
$P_1 \rightarrow P_0 \rightarrow M \rightarrow 0$
of $M$.  Set $s(M) = \mathrm{coker}(\mathrm{Hom}_{\mathcal{C}}(-,P_1) \rightarrow \mathrm{Hom}_{\mathcal{C}}(-,P_0))$.  This assignment extends to a right exact functor $s:\mathsf{mod}(R) \rightarrow \widehat{\mathcal{C}}$ such that $es = \mathrm{id}_{\mathsf{mod}(R)}$, as desired.
\end{proof}

\begin{rem}
One can alternatively deduce Theorem \ref{thm2} from a theorem of Schlichting.  According to \cite[Proposition 2]{Sch06}, $K(\widehat{\mathcal{C}}_0)$ is the homotopy fiber of $K(\mathrm{id}):K(\mathcal{C}^{\oplus}) \rightarrow K(\mathcal{C})$.  Now Theorem \ref{thm2} follows from the fact that $h:\mathcal{C}^{\oplus} \rightarrow \widehat{\mathcal{C}}$ and the inclusion $\mathcal{C} \rightarrow \mathsf{mod}(R)$ are $K$-theory equivalences (the former by Proposition \ref{prop1}, the latter by the resolution theorem).
\end{rem}

Next we classify the simple objects of $\widetilde{\mathcal{C}}$.  For each indecomposable maximal Cohen-Macaulay module $M$, let $R_M = \mbox{End}_R(M)^\mathrm{op}$.  Let $\kappa _M$ be the quotient of $R_M$ by its Jacobson radical; $\kappa _M$ is a division ring.

Since $\mathrm{End}_RM$ is a local ring, the functor $\mathcal{C}(-,M) \in \widetilde{\mathcal{C}}$ has a unique maximal subfunctor $F_M$, which coincides with $\mathrm{rad}(\mathcal{C})(-,M)$.  For any indecomposable $N \not \cong M$, $F_M(N) = \mathcal{C}(N,M)$, and $F_M(M) = \mathrm{rad}(\mathrm{End}_RM)$.

Let $S_M = \mathcal{C}(-,M)/F_M$.  As an additive functor, $S_M$ is determined uniquely up to isomorphism by the following properties.
\begin{itemize}
\item[1.]
$S_M (N) = 0$ if $N$ is indecomposable and not isomorphic to $M$.
\item[2.]
$S_M (M) = \kappa _M^{\mathrm{op}}$.
\item[3.]
For $f:M \longrightarrow M$, $S_M(f)(\alpha)=\alpha \cdot \overline{f}$, where $\overline{f}$ is the image of $f$ in $\kappa _M^{\mathrm{op}}$.
\end{itemize}

\begin{prop}
\label{sec:main-result}
Let $S$ be a functor in $\widetilde{\mathcal{C}}$.  Then the following are equivalent:
\begin{itemize}
\item[1.]
$S$ is simple in $\widetilde{\mathcal{C}}$, and $S(R) \cong 0$.
\item[2.]
$S \cong S_M$ for some indecomposable $R$-module $M$ which is not isomorphic to $R$.
\end{itemize}
\end{prop}

\begin{proof}
$2 \Rightarrow 1$ is clear, since $F_M \subset \mathcal{C}(-,M)$ is a maximal subfunctor and $F_M(R)=\mathcal{C}(R,M)$.  We prove $1 \Rightarrow 2$.  Suppose $S$ is a simple object in $\widetilde{\mathcal{C}}$.  $S$ cannot vanish on all indecomposables, so let $M$ be an indecomposable in $\mathcal{C}$ such that $S(M) \neq 0$.  Choose nonzero $x \in S(M)$.  $x$ defines a nonzero morphism $x_*:\mathcal{C}(-,M) \longrightarrow S$ by $(x_*)_N(f) = S(f)(x)$ for $f:N \longrightarrow M$ in $\mathcal{C}$.  As $S$ is simple, $x_*$ must be an epimorphism, and its kernel is a maximal subfunctor of $\mathcal{C}(-,M)$.  Therefore $S \cong \mathcal{C}(-,M)/F_M = S_M$.
\end{proof}

Recall that a map $f:E \rightarrow M$ in $\mathcal{C}$ is called \emph{right almost split} if $\mathrm{im}(\mathcal{C}(-,f))=\mathrm{rad}(\mathcal{C})(-,M)$.  If $M$ is indecomposable and not isomorphic to $R$, this is equivalent to saying that
\[ \xymatrix{
\mathcal{C}(-,E) \ar[r]^{\mathcal{C}(-,f)}	& \mathcal{C}(-,M) \ar[r]	& S_M \ar[r]	& 0
} \]
is exact.  $f$ is called \emph{minimal right almost split} if $f$ is right almost split and, for any $g \in \mathrm{End}_{\mathcal{C}}E$ such that $fg=f$, $g$ is an automorphism.

The following proposition was originally proved by Auslander (see, e.g., \cite{Aus86}) under the assumption that $R$ is complete; as observed in \cite[3.2]{Yos90}, one needs only assume $R$ is Henselian.  We omit the proof.
\begin{prop}
\label{existence}
For any indecomposable $M$ in $\mathcal{C}$ which is not isomorphic to $R$, there is a minimal right almost split morphism $E \longrightarrow M$.
\end{prop}
It follows that each functor $S_M$ is finitely presented.

The next theorem follows easily from results of Auslander; our proof is taken from \cite{Yos90}.

\begin{thm}
\label{thm3}
Suppose $R$ is of finite Cohen-Macaulay type.
Then every functor $F$ in $\widehat{\mathcal{C}}_0$ admits a filtration
\[
0 = F_0 \subset F_1 \subset \cdots \subset F_{n-1} \subset F_n = F
\]
with $F_i$ in $\widehat{\mathcal{C}}_0$ and $F_i / F_{i-1}$ simple in $\widetilde{\mathcal{C}}$, for all $i$.
\end{thm}

\begin{proof}
$F$ has a presentation of the form
\[ \xymatrix{
\mathcal{C}(-,N) \ar[r]^{\mathcal{C}(-,f)}	& \mathcal{C}(-,M) \ar[r]	& F \ar[r]	& 0
} \]
for some epimorphism $\xymatrix{f:N \ar@{>>}[r] & M}$ in $\mathcal{C}$.  Set $K = \mathrm{ker}(f)$, so $F$ is a subfunctor of $\mathrm{Ext}^1_R(-,K)$.  For any maximal Cohen-Macaulay module $L$ and prime $\mathfrak{p} \neq \mathfrak{m}$, $L_{\mathfrak{p}}$ is a maximal Cohen-Macaulay $R_{\mathfrak{p}}$-module; since $R_{\mathfrak{p}}$ is regular local, $L_{\mathfrak{p}}$ is in fact a free $R_{\mathfrak{p}}$-module.  Therefore $(\mathrm{Ext}^1_R (L,K))_{\mathfrak{p}} = \mathrm{Ext}^1_{R_\mathfrak{p}} (L_{\mathfrak{p}},K_{\mathfrak{p}}) = 0$.  Since $\mathrm{Ext}^1_R(L,K)$ is supported only at $\mathfrak{m}$, it must be a finite length $R$-module.  Therefore the submodule $F(L)$ is finite length as well.

Now let $L$ be the direct sum of all non-free indecomposables in $\mathcal{C}$.   The proof proceeds by induction on the length of $F(L)$.  If $\mathrm{length} (F(L)) = 0$, $F$ vanishes on all indecomposable maximal Cohen-Macaulay modules, so $F \cong 0$ and $F$ trivially admits the desired filtration.  Suppose, then, that $\mathrm{length} (F(L)) > 0$, and assume that any functor $G$ in $\widehat{\mathcal{C}}_0$ with $\mathrm{length}(G(L)) < \mathrm{length}(F(L))$ admits a filtration as above.  Choose an indecomposable $M$ with $F(M) \neq 0$; then choose an epimorphism of $R_M$-modules $p:F(M) \longrightarrow \kappa _M^{\mathrm{op}}$.  $p$ extends to a natural transformation $\pi :F \longrightarrow S_M$.  Let $G = \mathrm{ker}(\pi)$, so there is an exact sequence
\[
\xymatrix{
0 \ar[r]	& G \ar[r]	& F \ar[r]	& S_M \ar[r]	& 0
}.
\]
Therefore $\mathrm{length}(G(L)) < \mathrm{length}(F(L))$.  Since $\widehat{\mathcal{C}}_0$ is abelian, $G$ is in $\widehat{\mathcal{C}}_0$, so $G$ admits the desired filtration.  From the exact sequence above, it follows that $F$ admits such a filtration as well.
\end{proof}

\begin{defn}
Let $\widehat{\mathcal{C}}_0 ^s$ be the full subcategory of $\widehat{\mathcal{C}}_0$ consisting of objects which are semisimple in $\widetilde{\mathcal{C}}$.
\end{defn}

In other words, $\widehat{\mathcal{C}}_0^s$ consists of those functors which are isomorphic to finite direct sums of the functors $S_M$, with $M \not\cong R$.

\begin{defn}
We say $R$ is \emph{of finite Cohen-Macaulay type} if there are only finitely many isomorphism classes of indecomposable maximal Cohen-Macaulay $R$-modules.
\end{defn}

Theorem \ref{thm3} will imply the inclusion $\widehat{\mathcal{C}}_0^s \rightarrow \widehat{\mathcal{C}}_0$ is a $K$-theory equivalence; to prove this we first recall Quillen's D\'{e}vissage theorem.
\begin{thm}[{\cite[Thm. 4]{Qui73}}]
Let $\mathcal{A}$ be an abelian category and $\mathcal{B}$ a nonempty full subcategory closed under subobjects, quotients, and finite products in $\mathcal{A}$.  Suppose that every object $F$ in $\mathcal{A}$ admits a filtration
\[
0 = F_0 \subset F_1 \subset \cdots \subset F_{n-1} \subset F_n = F
\]
with $F_i/F_{i-1}$ in $\mathcal{B}$ for each $i$.

Then the inclusion functor induces a homotopy equivalence $K(\mathcal{B}) \simeq K(\mathcal{A})$.
\end{thm}

\begin{cor}
\label{cor1}
If $R$ is of finite Cohen-Macaulay type, the inclusion $\widehat{\mathcal{C}}_0 ^s \longrightarrow \widehat{\mathcal{C}}_0$ is a $K$-theory equivalence.
\end{cor}

\begin{proof}
Since subobjects, quotients, and products of semisimple objects in an abelian category are again semisimple, $\widehat{\mathcal{C}}_0^s$ is closed under taking subobjects, quotients, and products.  Theorem \ref{thm3} allows us to apply D\'{e}vissage to the subcategory $\widehat{\mathcal{C}}_0^s \subset \widehat{\mathcal{C}}_0$.  The conclusion follows.
\end{proof}

Assume $R$ is of finite Cohen-Macaulay type.  Let $\mathrm{ind}(\mathcal{C})$ denote the set of isomorphism classes of indecomposable objects in $\mathcal{C}$, and put $\mbox{ind}_0(\mathcal{C}) = \mbox{ind}(\mathcal{C}) \setminus \{[R]\}$.  $\widehat{\mathcal{C}}_0^s$ is semisimple, and its simple objects are the functors $S_M$ for $[M] \in \mbox{ind}_0(\mathcal{C})$.
Therefore the equivalences $\mathsf{proj}(\mathrm{End}(S_M)^{\mathrm{op}}) \simeq \mathsf{add}(S_M)$ induce an equivalence
\[
\bigoplus \limits_{[M] \in \mbox{ind}_0(\mathcal{C})} \mathsf{proj}(\mbox{End}(S_M)^{\mathrm{op}}) \xrightarrow{\sim} \widehat{\mathcal{C}}_0 ^s.
\]

Let $L=\bigoplus \limits_{[M] \in \mathrm{ind}(\mathcal{C})} M$ and put $\Lambda = (\mathrm{End}_RL)^{\mathrm{op}}$.  $\Lambda$ is sometimes called the \emph{Auslander algebra}.  Since $L$ is an additive generator for $\mathcal{C}^{\oplus}$, the horizontal functors in the diagram below are equivalences:
\[ \xymatrix @C=6pc {
\mathcal{C}^{\oplus} \ar[r]^{\mathrm{Hom}_{\mathcal{C}^{\oplus}}(L,-)}_{\simeq} \ar[d]^{h}	& \mathsf{proj}(\Lambda ) \ar@{^(->}[d] \\
\widehat{\mathcal{C}} \ar[r]^{F \mapsto F(L)}_{\simeq}				& \mathsf{mod}(\Lambda )
} \]

Combining everything, we end up with a diagram
\begin{equation}
\label{eq:5}
\xymatrix{
\bigoplus \limits_{\mbox{ind}_0(\mathcal{C})} \mathsf{proj}(\kappa _M) \ar[d]_{\simeq}	& 					& \mathsf{proj}(\Lambda ) \ar[dl]_{\simeq} \ar[d]^{{\simeq}_K}	& \\
\widehat{\mathcal{C}}_0^s \ar[d]_{\simeq_K}					& \mathcal{C}^\oplus \ar[d]_{\simeq _K}	& \mathsf{mod}(\Lambda) \ar[dl]_{\simeq}				& \\
\widehat{\mathcal{C}}_0 \ar[r]							& \widehat{\mathcal{C}} \ar[rr]		& 								& \mathsf{mod}(R) \\
}
\end{equation}
in which arrows labeled $\simeq_K$ induce equivalences in $K$-theory.  Together with the $K$-theory equivalences in this diagram, Theorem \ref{thm2} implies there are homotopy fiber sequences
\begin{equation} \label{eq:3} \xymatrix{
\bigvee \limits_{[M] \in \mathrm{ind}_0(\mathcal{C})} K(\kappa _M) \ar[r]	& K'(\Lambda ) \ar[r]	& K'(R)
} \end{equation}
\begin{equation} \label{eq:4} \xymatrix{
\bigvee \limits_{[M] \in \mathrm{ind}_0(\mathcal{C})} K(\kappa _M) \ar[r]^-{\alpha}	& K(\mathcal{C}^{\oplus}) \ar[r]^{\beta}	& K'(R)
} \end{equation}
Taking homotopy groups in (\ref{eq:3}) yields the long exact sequence of Theorem \ref{sequencetheorem}.

\begin{rem}
\label{algclosed}
By Nakayama's Lemma, the image of $\mathfrak{m}$ in $R_M$ is contained in the maximal ideal of $R_M$.  So we may view $\kappa _M$ as a division algebra over $R/\mathfrak{m}$.  As $R_M$ is a finitely generated $R$-module, $\kappa _M$ is a finite dimensional vector space over $R/\mathfrak{m}$.  In particular, if $R/\mathfrak{m}$ is algebraically closed, $\kappa _M = R/\mathfrak{m}$.
\end{rem}

\begin{rem}
\label{mf}
Suppose $R$ is of the form $S/(w)$ for some regular local ring $S$ and $w \in S$.  Then we may apply the techniques above to obtain a decomposition of the $K$-theory of the category $\mathsf{MF}$ of matrix factorizations in $S$ with potential $w$.  An object of this category is a $\mathbb{Z}/2 \mathbb{Z}$-graded finitely generated free $S$-module $X$ with a degree-one endomorphism $d_X$ such that $d_X^2 = w \cdot \mathrm{id}$.   A morphism $f:X \rightarrow Y$ in $\mathsf{MF}$ is a degree zero map satisfying $d_Y f = f d_X$.

$\mathsf{MF}$ is a Frobenius category whose subcategory $\mathsf{prinj}(\mathsf{MF})$ of projective-injective objects consists of the contractible matrix factorizations -- that is, those objects $X$ for which there is a degree-one endomorphism $t:X \rightarrow X$ such that $td_X + d_X t = \mathrm{id}_X$.  Any Frobenius category $\mathcal{F}$ defines an exact category with weak equivalences $w\mathcal{F}$ consisting of those morphisms becoming invertible in the stable category $\mathcal{F}/\mathsf{prinj}(\mathcal{F})$.  We shall take the $K$-theory of $\mathsf{MF}$ relative to this subcategory $w\mathsf{MF}$ of weak equivalences.

The category $\mathsf{Ch}^b(\mathcal{E})$ of bounded chain complexes in an exact category $\mathcal{E}$ is a Frobenius category; its conflations are the sequences which are degree-wise split, and its projective-injective objects are the contractible complexes.  Denote by $\mathsf{perf}(R)$ the category of perfect complexes of $R$-modules, i.e. the exact subcategory of $\mathsf{Ch}^b\mathsf{mod}(R)$ consisting of complexes quasi-isomorphic to a complex of free $R$-modules.  There is a map of Frobenius pairs
\[
\xymatrix{
\Omega :(\mathsf{MF},\mathsf{prinj}(\mathsf{MF})) \ar[r]	& (\mathsf{Ch}^b\mathsf{mod}(R),\mathsf{perf}(R))
} \] taking a matrix factorization $\xymatrix{X^1 \ar@<+.5ex>[r]^{d^1} & \ar@<+.5ex>[l]^{d^0} X^0}$ to $\mathrm{coker}(d^1)$, considered as a complex concentrated in degree zero.  (Since $\mathrm{coker}(d^1)$ is annihilated by $w$, we may view it as an $R$-module.)  By ~\cite[Theorem 3.9]{Orl03}, $\Omega$ induces an equivalence on derived categories, so in particular it is an equivalence in $K$-theory (\cite[Proposition 3]{Sch06}).  Since the degree-zero inclusions $\mathsf{mod}(R) \rightarrow \mathsf{Ch}^b\mathsf{mod}(R)$ and  $\mathsf{proj}(R) \rightarrow \mathsf{Ch}^b\mathsf{proj}(R)$ are $K$-theory equivalences (\cite[1.11.7]{TT90}), and since the inclusion $\mathsf{Ch}^b\mathsf{proj}(R) \rightarrow \mathsf{perf}(R)$ is a derived equivalence, it follows that $K(\mathsf{MF})$ is equivalent to the homotopy cofiber of the map $K(R) \longrightarrow K'(R)$ induced by the inclusion $\mathsf{proj}(R) \rightarrow \mathsf{mod}(R)$.

Let $r:\mathsf{proj}(R) \rightarrow \mathcal{C}^{\oplus}$ be the inclusion, and set $X = \mathrm{cone}(K(r))$.  Consider the following exact triangles of spectra.
\begin{equation}
 \xymatrix @R=0.2pc @C=3pc {
  K(R) \ar[r]^{K(r)}			& K(\mathcal{C}^{\oplus}) \ar[r]^{\rho}	& X \ar[r]											& \sum K(R) \\
  K(\mathcal{C}^{\oplus}) \ar[r]^{\beta}	& K'(R) \ar[r]				& \sum \bigvee \limits_{\mathrm{ind}_0(\mathcal{C})}^{\phantom{t}} K(\kappa _M) \ar[r]^-{\sum \alpha}	& \sum K(\mathcal{C}^{\oplus}) \\
  K(R) \ar[r]^{\beta \circ K(r)}		& K'(R) \ar[r]				& K(\mathsf{MF}) \ar[r]										& \sum K(R) \\
 }
 \label{eq:2}
\end{equation}The middle sequence is the triangle from (\ref{eq:4}), rotated once.  Using the octahedral axiom to compare the cones of $\beta$, $K(r)$, and $\beta \circ K(r)$, we obtain an exact triangle
\begin{equation} \label{MFtriangle}
 \xymatrix{
  \bigvee \limits_{\mathrm{ind}_0(\mathcal{C})}^{\phantom{t}} K(\kappa _M) \ar[r]^-{\alpha '}	& X \ar[r]	& K(\mathsf{MF}) \ar[r]	& \sum \bigvee \limits_{\mathrm{ind}_0(\mathcal{C})}^{\phantom{t}} K(\kappa _M)
 }.
\end{equation}
We will study $\alpha '$ in the next section.
\end{rem}

\section{The Auslander-Reiten Matrix}
\label{ausl-reit-matr}
Assume now that $R$ is of finite CM type, that $k=R/\mathfrak{m}$ is algebraically closed, and that $R$ contains $k$.  We wish to understand the map
\begin{equation} \label{Kchain}\xymatrix{
\alpha:\bigvee \limits_{\mathrm{ind}_0(\mathcal{C})}K(k) \ar[r]	& K(\mathcal{C}^{\oplus})
} \end{equation}
which appears in (\ref{eq:4}).

Let $M$ be an indecomposable in $\mathcal{C}$ which is not isomorphic to $R$.  A short exact sequence
\begin{equation}
\label{arsequence}
\xymatrix{
0 \ar[r]	& N \ar[r]^g	& E \ar[r]^f	& M \ar[r]	& 0 \\
}
\end{equation}
in $\mathcal{C}$ is called the \emph{Auslander-Reiten sequence} ending in $M$ if $f$ is minimal right almost split.  This is the same as saying that
\[ \xymatrix{
0 \ar[r]	& \mathcal{C}(-,N) \ar[r]^{\mathcal{C}(-,g)}	& \mathcal{C}(-,E) \ar[r]^{\mathcal{C}(-,f)}	& \mathcal{C}(-,M)
} \]
is a minimal projective resolution in $\widehat{\mathcal{C}}$ of $S_M$.  The Auslander-Reiten sequence ending in $M$ is unique up to isomorphism.  By Proposition \ref{existence}, any non-projective indecomposable $M$ in $\mathcal{C}$ is part of an Auslander-Reiten sequence as in (\ref{arsequence}).

Let $M_0,\ldots ,M_n$ be the indecomposable maximal Cohen-Macaulay $R$-modules, with $M_0 = R$.  For $j>0$, set
\[ \xymatrix{0 \ar[r] & N_j \ar[r] & E_j \ar[r] & M_j \ar[r] & 0 }\]
to be the Auslander-Reiten sequence ending in $M_j$.  Given any $Q$ in $\mathcal{C}$, let $\#(j,Q)$ be the number of $M_j$-summands appearing in a decomposition of $Q$ into indecomposables.  Denote by $k_j$ the object of $\bigoplus \limits_{\mathrm{ind}_0(\mathcal{C})}\mathsf{mod}(k)$ which is $k$ in the $M_j$ coordinate and $0$ in the others.  Note that to define a $k$-linear functor out of $\bigoplus \limits_{\mathrm{ind}_0(\mathcal{C})}\mathsf{mod}(k)$, one needs only to specify the image of each object $k_j$.

Define $k$-linear functors
\begin{align*}
a_0,a_1,a_2:\bigoplus \limits_{\mathrm{ind}_0(\mathcal{C})} \mathsf{mod}(k) & \rightarrow	\mathcal{C}^{\oplus}
\end{align*}
by
\begin{align*}
a_0(k_j) &= M_j	& a_1(k_j)& = E_j				& a_2(k_j)&= N_j
\end{align*}
Set $a:\bigoplus \limits_{\mathrm{ind}_0(\mathcal{C})} \mathsf{mod}(k) \rightarrow \widehat{\mathcal{C}}$ to be the $k$-linear functor sending $k_j$ to $S_{M_j}$, and as before let $h:\mathcal{C}^{\oplus} \rightarrow \widehat{\mathcal{C}}$ be the Yoneda functor.  Tracing through the functors in (\ref{eq:5}), one sees that $K(a)=K(h) \circ \alpha$.

Since there is an exact sequence of functors
\[ \xymatrix{
0 \ar[r]	& h \circ a_2 \ar[r]	& h \circ a_1 \ar[r]	& h \circ a_0 \ar[r]	& a \ar[r]	& 0,
} \]
we have by the additivity theorem (\cite[$\S$3, Cor. 2]{Qui73})
\[K(ha_0)-K(ha_1)+K(ha_2) = K(a)\]
so that
\begin{equation}
\label{alpha}
\alpha =K(h)^{-1}K(a)=K(a_0)-K(a_1)+K(a_2).
\end{equation}
Let $m_l:\mathsf{mod}(k) \rightarrow \mathcal{C}^{\oplus}$ be the $k$-linear functor which sends $k$ to $M_l$.  Form an $(n+1) \times n$ integer matrix $T$ whose $lj$-entry is $\#(l,M_j)-\#(l,E_j)+\#(l,N_j)$.  ($T$ has a 0th row but no 0th column.)  Applying the additivity theorem to $a_0$, $a_1$, and $a_2$, we conclude from (\ref{alpha}) that the $j$th component of $\alpha$ is
\begin{equation} \label{alpha2}(\alpha )_j = \sum \limits_l T_{lj}K(m_l). \end{equation}
Let $m=\bigoplus m_l:\bigoplus \limits_{\mathrm{ind}(\mathcal{C})} \mathsf{mod}(k) \rightarrow \mathcal{C}^{\oplus}$.  Then we may rewrite (\ref{alpha2}) as
\begin{equation}
\label{ARmatrix}
\alpha = K(m) \circ (T \cdot \mathrm{id}_{K(k)})
\end{equation}
In this sense $\alpha $ is defined by the matrix $T$.

Identifying $K_0(\mathcal{C}^{\oplus})$ with $\mathbb{Z}^{n+1}$ via the basis $\{[M_0],\ldots ,[M_n]\}$, we see that $K_0(m_l):K_0(\mathsf{mod}(k)) = \mathbb{Z} \rightarrow \mathbb{Z}^{n+1}$ is just the inclusion into the $l$th coordinate.  Therefore $\pi _0(\alpha)$, as a map between free abelian groups, is defined by $T$.  This is the description of $\pi _0(\alpha)$ originally given in \cite[$\S$ 4.3]{AusRei86}.

Let $T'$ be the $n \times n$ integer matrix obtained from $T$ by deleting its top row, which corresponds to the indecomposable $M_0=R$.  Just as $T$ described $\alpha$, $T'$ describes the map $\alpha ' = \rho \circ \alpha :\bigvee \limits_{\mathrm{ind}_0(\mathcal{C})}K(k) \rightarrow X$ from (\ref{MFtriangle}).  Recall from (\ref{eq:2}) the homotopy fiber sequence
\[ \xymatrix{
K(R) \ar[r]^{K(r)}	& K(\mathcal{C}^{\oplus}) \ar[r]^{\rho}	& X
} \]
Since $m_0:\mathsf{mod}(k) \rightarrow \mathcal{C}^{\oplus}$ factors through $r$, $\rho \circ K(m_0)$ is nulhomotopic and therefore the $j$th component of $\alpha '$ is given by
\begin{align}
(\alpha ')_j	& = \rho \circ \alpha _j \nonumber \\
		& = \rho \circ \sum \limits_{l \geq 0} T_{lj} K(m_l) \nonumber \\
		& = \rho \circ \sum \limits_{l>0} T'_{lj} K(m_l) \label{alpha'}
\end{align}
Let $m'=\bigoplus \limits_{l>0} m_l:\bigoplus \limits_{\mathrm{ind}_0(\mathcal{C})} \mathsf{mod}(k) \rightarrow \mathcal{C}^{\oplus}$.  Then we may rewrite (\ref{alpha'}) as
\begin{equation}
\label{ARmatrix'}
\alpha' = \rho \circ K(m') \circ (T' \cdot \mathrm{id}_{K(k)})
\end{equation}

\section{An Example}
\label{an-example}
Let $R$ be a 1-dimensional singularity of type $A_{2n}$, i.e. $R=k[[t^2,t^{2n+1}]]$ with $k$ an algebraically closed field.  The MCM $R$-modules are the modules $M_i=k[[t^2,t^{2(n-i)+1}]]$, $i=0,\ldots ,n$, on which $R$ acts by multiplication.  The Auslander-Reiten quiver of $R$ is then
\begin{equation} \label{ARquiver} \xymatrix{
[R] \ar@<.5ex>[r]	& [M_1] \ar@<.5ex>[l]^{t^2} \ar@<.5ex>[r] \ar@(dr,dl)@{..}[]	& \cdots \ar@<.5ex>[l]^{t^2} \ar@<.5ex>[r]	& [M_n] \ar@<.5ex>[l]^{t^2} \ar@(ur,dr)[]^t \ar@(dr,dl)@{..}[]
} \end{equation}
(each right arrow is the inclusion $x \mapsto x$). 
In this section we will try to describe, as explicitly as possible, the groups $K_1'(R)$ and $K_1(\mathsf{MF})$, using the techniques developed elsewhere in this paper.  These descriptions appear in Proposition \ref{decompositions}.

Let $\mathcal{B}_i=\mathsf{add}(M_0,\ldots ,M_i) \subset \mathcal{C}^{\oplus}$, and let $\mathcal{B}_{-1}=\{0\} \subset \mathcal{C}^{\oplus}$.  Let $f_i:\mathcal{B}_{i-1} \rightarrow \mathcal{B}_{i}$ denote the inclusion functor and $p_i:\mathcal{B}_i \rightarrow \mathcal{B}_i/\mathcal{B}_{i-1}$ the quotient functor; let $F_i$ denote the image in $K_1(\mathcal{C}^{\oplus})$ of $K_1(\mathcal{B}_i)$.  Then the solid diagram below commutes and has exact rows; the top row is exact by Theorem \ref{K1}.
\begin{equation} \label{K1diagram} \xymatrix{
		& K_1(\mathcal{B}_{i-1}) \ar[r]^{K_1(f_i)} \ar@{>>}[d]	& K_1(\mathcal{B}_i) \ar[r]^-{K_1(p_i)} \ar@{>>}[d]		& K_1(\mathcal{B}_i/\mathcal{B}_{i-1}) \ar[r] \ar[d]	& 0 \\
0 \ar[r]	& F_{i-1} \ar[r]						& F_i \ar[r] \ar@{-->}[ru]					& F_i/F_{i-1} \ar[r]				& 0
} \end{equation}
Moreover, in the diagram below, the right vertical arrow is an equivalence when $i>0$.
\[ \xymatrix{
\mathcal{B}_i \ar[r] \ar[d]	& \mathcal{B}_i/\mathcal{B}_{i-1} \ar[d]^{\simeq} \\
\mathcal{C}^{\oplus} \ar[r]	& {\mathcal{C}^{\oplus}}/\mathsf{add}(\bigoplus _{l \neq i} M_l)
} \]
Therefore for $i>0$, the map $K_1(\mathcal{B}_i) \rightarrow K_1(\mathcal{B}_i/\mathcal{B}_{i-1})$ factors through $F_i$ as indicated in (\ref{K1diagram}).  It follows that the right vertical map in (\ref{K1diagram}) is an isomorphism.  So there are short exact sequences
\begin{equation} \label{K1sequence} \xymatrix{
0 \ar[r]	& F_{i-1} \ar[r]	& F_i \ar[r]	& K_1(\mathcal{B}_i/\mathcal{B}_{i-1}) \ar[r]	& 0
} \end{equation}

For each $i$, $\mathcal{B}_{i}/\mathcal{B}_{i-1}$ has one nonzero indecomposable $M_{i}$.  The ring homomorphism $k[[t^2,t^{2(n-i)+1}]] \rightarrow \mathrm{End}_R(M_i)$, sending $f$ to the multiplication-by-$f$ endomorphism, is an isomorphism for each $i$.  Using this and the AR quiver (\ref{ARquiver}), we see that for $i>0$,
\[
\mathrm{End}_{\mathcal{B}_{i}/\mathcal{B}_{i-1}}M_i = (\mathrm{End}_R M_i)/(t^2)= \left\{
\begin{array}{l l}
k		& \mbox{if } 0<i<n \\
k[t]/(t^2)	& \mbox{if } i=n
\end{array}\right .
\]
so that $\mathcal{B}_i/\mathcal{B}_{i-1} \simeq \mathsf{mod}(k)$ if $0<i<n$, and $\mathcal{B}_n/\mathcal{B}_{n-1} \simeq \mathsf{proj}(k[t]/(t^2))$.  Let $k^+$ be the additive abelian group of $k$.  Then $(k[t]/(t^2))^{\times} \cong k^{\times} \oplus k^+$ via the identification $\alpha(1 + \beta t) \mapsto (\alpha, \beta )$.  Therefore
\[
K_1(\mathcal{B}_i/\mathcal{B}_{i-1}) = K_1(\mathrm{End}_{\mathcal{B}_i/\mathcal{B}_{i-1}}M_i) = \left\{
\begin{array}{l l}
k^{\times}		& \mbox{if } 0<i<n \\
k^{\times} \oplus k^+	& \mbox{if } i=n
\end{array}
\right .
\]

So, according to the sequences (\ref{K1sequence}) and the descriptions above of the groups $K_1(\mathcal{B}_i/\mathcal{B}_{i-1})$, there is a filtration
\[ 0 \subset F_0 \subset \cdots \subset F_n = K_1(\mathcal{C}^{\oplus})\]
such that
\begin{itemize}
\item[1.]
$F_0$ is a quotient of $K_1(R)=R^{\times}$.
\item[2.]
$F_i/F_{i-1} \cong k^{\times}$ for $i=1, \ldots ,n-1$.
\item[3.]
$F_n/F_{n-1} \cong k^{\times} \oplus k^+$.
\end{itemize}
The group $k^{\times}$ appears in this filtration as a subquotient of $K_1(\mathcal{C}^{\oplus})$ $n+1$ times: it appears as a subobject of $F_0$ (we shall soon see that the composition $k^{\times} \rightarrow R^{\times} \rightarrow F_0$ is monic); it appears $n-1$ times as $F_i/F_{i-1}$, $0<i<n$; and it appears as a summand of $F_n/F_{n-1}$.  We next argue that each of these copies of $k^{\times}$ is in fact a summand of $K_1(\mathcal{C}^{\oplus})$.

Let
\[\begin{array}{rcl} m_i:\mathsf{mod}(k) \rightarrow \mathcal{C}^{\oplus}& \mbox{and} & j_i:\mathsf{mod}(k) \rightarrow \mathcal{B}_i\end{array}\] be the $k$-linear functors which (both) send $k$ to $M_i$, and let 
\[\begin{array}{rcl}m=\bigoplus m_i:(\mathsf{mod}(k))^{\oplus n+1} \rightarrow \mathcal{C}^{\oplus}& \mbox{and} & m'=\bigoplus \limits_{i>0} m_i:(\mathsf{mod}(k))^{\oplus n} \rightarrow \mathcal{C}^{\oplus}.\end{array}\]
Let $q:\mathcal{C}^{\oplus} \rightarrow \mathcal{C}^{\oplus}/\mathrm{rad}(\mathcal{C}^{\oplus})$ be the quotient functor and
\[\mathrm{proj}_i:\mathcal{C}^{\oplus}/\mathrm{rad}(\mathcal{C}^{\oplus}) \simeq (\mathsf{mod}(k))^{\oplus n+1} \rightarrow \mathsf{mod}(k)\]
the $i$th projection.  Then the diagram below commutes.
\[ \xymatrix{
\mathsf{mod}(k) \ar[r]^= \ar[d]_{j_i} \ar@/_2pc/[dd]_{m_i}	& \mathsf{mod}(k) \\
\mathcal{B}_i \ar[r]^{p_i} \ar@{^(->}[d]				& \mathcal{B}_i/\mathcal{B}_{i-1} \ar[d] \\
\mathcal{C}^{\oplus} \ar[r]^q					& \mathcal{C}^{\oplus}/\mathrm{rad}(\mathcal{C}^{\oplus}) \ar@/_2pc/[uu]_{\mathrm{proj}_i}
} \]
From this we deduce the following.
\begin{itemize}
\item[1.]
The map $k^{\times} \rightarrow F_0$ induced by $j_0$ has a left inverse which factors through $K_1(\mathcal{C}^{\oplus})$.  Therefore $k^{\times}$ is embedded in $F_0$ in such a way that is a summand of $K_1(\mathcal{C}^{\oplus})$.
\item[2.]
For $0<i<n$, $K_1(j_i)$ embeds $k^{\times}$ as a summand of $K_1(\mathcal{C}^{\oplus})$ which is contained in $F_i$ and projects isomorphically onto $F_i/F_{i-1}$.
\item[3.]
$K_1(j_n)$ embeds $k^{\times}$ as a summand of $K_1(\mathcal{C}^{\oplus})$, and $K_1(p_n j_n):k^{\times} \rightarrow k^{\times} \oplus k^+$ is the inclusion into the first coordinate.
\end{itemize}

We compile all of this data in the following commuting diagram, in which all rows are split short exact sequences and all columns are exact.
\begin{equation} \label{decompositiondiagram} \xymatrix{
0 \ar[r]	& k^{\times} \ar[r] \ar[d] \ar[dr]^{K_1(m_0)}	& R^{\times} \ar[r] \ar[d]^{K_1(r)}			& R^{\times}/k^{\times} \ar[r] \ar[d]	& 0 \\
0 \ar[r]	& (k^{\times})^{n+1} \ar[r]^{K_1(m)} \ar[d]		& K_1(\mathcal{C}^{\oplus}) \ar[r] \ar[d]^{\pi _1(\rho)}	& \mathrm{coker}\,K_1(m) \ar[r] \ar[d]	& 0 \\
		& (k^{\times})^n \ar[d]				& K_1(X) \ar[d]						& k^+ \ar[d]				& \\
		& 0						& 0							& 0					&
} \end{equation}
Note that $\pi _1(\rho)$ is surjective because there is an exact sequence
\[ \xymatrix @R=0.5pc {
K_1(\mathcal{C}^{\oplus}) \ar[r]^{\pi _1(\rho)}	& K_1(X) \ar[r]	& K_0(R) \ar@{=}[d] \ar[r]	& K_0(\mathcal{C}^{\oplus}) \ar@{=}[d] \\
						& 		& \mathbb{Z} \ar@{^(->}[r]	& \mathbb{Z}^{n+1}
} \]
Therefore the terms in the third row of (\ref{decompositiondiagram}) fit into an exact sequence
\[ \xymatrix{
(k^{\times})^n \ar[rr]^{\pi _1(\rho) \circ K_1(m')}	& & K_1(X) \ar[r]	& k^+ \ar[r]	& 0.
} \]
We next argue that the first map in this sequence is injective.  For this it suffices to show that $\mathrm{im}\,K_1(m) \cap \mathrm{im}\,K_1(r) \subset \mathrm{im}\,K_1(m_0)$.  For $i \neq 0$, the composition $\xymatrix{\mathsf{proj}(R) \ar[r]^r	& \mathcal{C}^{\oplus} \ar[r]^-{\mathrm{proj}_i \circ q}	& \mathsf{mod}(k) }$ is zero, so $\mathrm{im}\,K_1(r) \subset \mathrm{ker}\,K_1(\mathrm{proj_i} \circ q)$ and therefore
\begin{align*}
\mathrm{im}\,K_1(m) \cap \mathrm{im}\,K_1(r)	& \subset \mathrm{im}\,K_1(m) \cap (\bigcap _{i>0} \mathrm{ker}\,K_1(\mathrm{proj}_i \circ q)) \\
						& = K_1(m)(\bigcap _{i>0}\mathrm{ker}\,K_1(\mathrm{proj}_i \circ q \circ m)) \\
						& = \mathrm{im}\,K_1(m_0),\end{align*}
as desired.

Now using (\ref{eq:4}) and (\ref{ARmatrix}) one obtains an exact sequence
\begin{equation} \label{K1'sequence} \xymatrix @R=0.7pc {
(k^{\times})^n \ar[rr]^-{K_1(m) \circ (T \cdot \mathrm{id}_{k^{\times}})}	& & K_1(\mathcal{C}^{\oplus}) \ar[r]	& K_1'(R) \ar[r]	& (K_0(k))^n \ar[r] \ar@{=}[d]	& K_0(\mathcal{C}^{\oplus}) \ar@{=}[d] \\
									& & 					& 			& \mathbb{Z}^{n} \ar[r]^{T}	& \mathbb{Z}^{n+1}
} \end{equation}
Similarly, by (\ref{MFtriangle}) and (\ref{ARmatrix'}) there is an exact sequence
\begin{equation} \label{K1MFsequence} \xymatrix @R=0.7pc {
(k^{\times})^n \ar[rrr]^-{\pi _1 (\rho) \circ K_1(m') \circ (T' \cdot \mathrm{id}_{k^{\times}})}	& & & K_1(X) \ar[r]	& K_1(\mathsf{MF}) \ar[r]	& (K_0(k))^n \ar[r] \ar@{=}[d]	& K_0(X) \ar@{=}[d] \\
												& & & 			&	 			& \mathbb{Z}^{n} \ar[r]^{T'}	& \mathbb{Z}^{n}
} \end{equation}
The matrices $T$ and $T'$ can be computed directly from the Auslander-Reiten quiver (\ref{ARquiver}); keeping in mind our convention that $T$ has a zeroth row but no zeroth column, these matrices are
\[ T = \left( \begin{array}{rrrcrr}
-1	& 0		& 0	& 	& 	& \\
2	&- 1		& 0	& 	& 	& \\
-1	& 2		& -1	& \cdots& 	& \\
0	& -1		& 2	& 	& 0	& \\
0	& 0		& -1	& 	& -1	& 0 \\
0	& 0		& 0	& \ddots& 2	& -1 \\
	& \vdots	& 	& 	& -1	& 1
\end{array} \right) \mbox{; explicitly, }T_{lj} = \left\{ \begin{array}{l l}
-1	& \mbox{if } j=l \pm 1 \\
2	& \mbox{if } j=l<n \\
1	& \mbox{if } j=l=n \\
0	& \mbox{otherwise}
\end{array}\right . \]
\[ T' = \left( \begin{array}{rrrcrr}
2	&- 1		& 0	& 	& 	& \\
-1	& 2		& -1	& \cdots& 	& \\
0	& -1		& 2	& 	& 0	& \\
0	& 0		& -1	& 	& -1	& 0 \\
0	& 0		& 0	& \ddots& 2	& -1 \\
	& \vdots	& 	& 	& -1	& 1
\end{array} \right) . \]
One proves easily by induction that $\mathrm{det}\,T'>0$, so the last map in each sequence (\ref{K1'sequence}) and (\ref{K1MFsequence}) is injective.  Since these sequences are exact, it follows that
\begin{align*}
K_1'(R) & \cong \mathrm{coker}[K_1(m) \circ (T \cdot \mathrm{id}_{k^{\times}})] \\
K_1(\mathsf{MF}) & \cong \mathrm{coker}[\pi _1(\rho) \circ K_1(m') \circ (T' \cdot \mathrm{id}_{k^{\times}})]
\end{align*}
Together with the data from (\ref{decompositiondiagram}), we obtain the following decompositions. \vspace{1.6pc}
\begin{prop}
\
\label{decompositions}
\begin{itemize}
\item[1.]
There is an abelian group $G\ (=\mathrm{coker}\,K_1(m))$ such that
\[
K_1(\mathcal{C}^{\oplus}) \cong \mathrm{coker}(T\cdot \mathrm{id}_{k^{\times}}) \oplus G
\]
and $G$ fits into an exact sequence
\[ \xymatrix{
R^{\times}/k^{\times} \ar[r]	& G \ar[r]	& k^+ \ar[r]	& 0.
} \]
\item[2.]
There is a short exact sequence
\[ \xymatrix{
0 \ar[r]	& \mathrm{coker}(T'\cdot \mathrm{id}_{k^{\times}}) \ar[r]	& K_1(\mathsf{MF}) \ar[r]	& k^+ \ar[r]	& 0.
} \]
\end{itemize}
\end{prop}

\appendix
\section{\texorpdfstring{$K_1\ $}{
}%
 Localization for a Krull-Schmidt Category}
\label{loca-for-a-krul-schm-cate}
\subsection*{$K_1$ of an Additive Category}
Let $\mathcal{A}$ be an additive category.  We shall consider $\mathcal{A}$ to be an exact category with the split exact structure, in which the conflations are the sequences isomorphic to a direct sum sequence $\xymatrix @C=0.8pc{A\, \ar@{>->}[r] & A \oplus B \ar@{->>}[r] & B}$.  Set $\mathrm{Aut}(\mathcal{A})$ to be the category whose objects are pairs $(A, \phi)$ with $A$ in $\mathcal{A}$ and $\phi \in \mathrm{Aut}_{\mathcal{A}}A$, and whose morphisms are defined by
\[\mathrm{Hom}_{\mathrm{Aut}(\mathcal{A})}((A,\phi ),(A',\phi ')) = \{f \in \mathrm{Hom}_{\mathcal{A}}(A,A') | \phi ' f = f \phi\}.\]
$\mathrm{Aut}(\mathcal{A})$ has an exact structure in which a sequence is a conflation iff it is a conflation in $\mathcal{A}$.  Recall from (\cite[$\S$3]{She82}) that there is a natural surjection \[\xymatrix {K_0(\mathrm{Aut}(\mathcal{A})) \ar@{>>}[r] & K_1(\mathcal{A})}\] whose kernel is generated by elements of the form $[(A,\alpha \beta)] - [(A,\alpha)] - [(A, \beta)]$.  Denote by $[A,\phi ]$, or just $[\phi]$, the image in $K_1(\mathcal{A})$ of the $K_0$-class of the object $(A, \phi)$ of $\mathrm{Aut}(\mathcal{A})$.

\begin{lem}
Let $f:A \rightarrow A'$ be a morphism in $\mathcal{A}$, and let $\phi = \left( \begin{smallmatrix}1_A & 0 \\ f & 1_{A'} \end{smallmatrix} \right) \in \mathrm{Aut}_{\mathcal{A}}(A \oplus A')$.  Then in $K_1(\mathcal{A})$, $[\phi ] = 0$.  Similarly, given any $g:A' \rightarrow A$, $\left[ \left( \begin{smallmatrix} 1_A & g \\ 0 & 1_{A'}\end{smallmatrix} \right) \right] = 0$.
\end{lem}

\begin{proof}
There is a conflation in $\mathrm{Aut}(\mathcal{A})$
\[ \xymatrix{
(A',1_{A'})\ \ar@{>->}[r]	& (A \oplus A', \phi ) \ar@{>>}[r]	& (A,1_A)
} \]
so that $[\phi ] = [1_{A'}] + [1_A] = 0$.  The second statement is proved the same way.
\end{proof}

\begin{rem}
\label{reduction}
Suppose $\phi$, $\psi \in \mathrm{Aut}_{\mathcal{A}}(\bigoplus A_i)$ are row- or column-equivalent -- that is, the matrices defining $\phi$ and $\psi$ differ only up to left- or right-multiplication by elementary matrices, which are identity along the diagonal and zero off the diagonal except in one entry.  Then using the lemma above, one sees easily that $[\phi] = [\psi]$.
\end{rem}
\subsection*{Automorphisms in a Krull-Schmidt Category}

The following properties of $\mathrm{rad}(\mathcal{A})$ are easy to prove.
\begin{itemize}
\item[1.]
$\mathrm{Hom}_{\mathrm{rad}(\mathcal{A})}(X,X)$ is the Jacobson radical $\mathrm{rad}(\mathrm{End}_{\mathcal{A}}(X))$ of $\mathrm{End}_{\mathcal{A}}(X)$.
\item[2.]
If $X=\bigoplus X_i$, $Y=\bigoplus Y_j$, then $\mathrm{Hom}_{\mathrm{rad}(\mathcal{A})}(X,Y)=\bigoplus \limits_{i,j} \mathrm{Hom}_{\mathrm{rad}(\mathcal{A})}(X_i,Y_j)$.
\item[3.]
If $\mathcal{A}$ is Krull-Schmidt and $X$ and $Y$ are nonisomorphic indecomposables, $\mathrm{Hom}_{\mathrm{rad}(\mathcal{A})}(X,Y)=\mathrm{Hom}_{\mathcal{A}}(X,Y)$.
\end{itemize}
The next lemma gives us an easy criterion for recognizing automorphisms in a Krull-Schmidt additive category.
\begin{lem}
\label{invertibility-criterion}
Suppose $\mathcal{A}$ is Krull-Schmidt and $A = \bigoplus \limits_{i=1}^d A_i ^{n_i}$ with $A_1, \ldots , A_d$ pairwise nonisomorphic indecomposables in $\mathcal{A}$.  Suppose $\phi \in \mathrm{End}_{\mathcal{A}}A$, and denote by $\phi _{ij}$ the induced morphism $A_j^{n_j} \rightarrow A_i^{n_i}$.  Then $\phi$ is invertible iff for all $i$, $\phi _{ii}$ is invertible.
\end{lem}
\begin{proof}
Let $R=\mathrm{End}_{\mathcal{A}}A$ and $R_i = \mathrm{End}_{\mathcal{A}}(A_i^{n_i})$.  Using the above properties of $\mathrm{rad}(\mathcal{A})$, one sees that the Jacobson radical of $R$ is 
\[
\mathrm{rad}(R) = \mathrm{Hom}_{\mathrm{rad}(\mathcal{A})}(\oplus A_i^{n_i} , \oplus  A_j^{n_j}) = \{\phi \in R | \phi_{ii} \in \mathrm{rad}(R_i) \mbox{ for all }i\},
\]
so that $R/\mathrm{rad}(R) = \prod R_i/\mathrm{rad}(R_i)$.  Therefore $\phi$ is invertible iff $\overline{\phi}$ is invertible in $R/\mathrm{rad}(R)$, iff $\overline{\phi _{ii}}$ is invertible in $R_i/\mathrm{rad}(R_i)$ for all $i$, iff $\phi _{ii}$ is invertible for all $i$.
\end{proof}
Using this lemma we can deduce the following.
\begin{cor}
\label{sec:2}
Assume $\mathcal{A}$ is Krull-Schmidt, and let $\mathcal{B} \subset \mathcal{A}$ be a full additive category closed under summands.  Let $A$ be an object of $\mathcal{A}$ which has no nonzero summand in $\mathcal{B}$, and let $B$ be an object of $\mathcal{B}$.  Let $\phi = \left(\begin{smallmatrix} a & b \\ c & d \end{smallmatrix}\right) \in \mathrm{Aut}_{\mathcal{A}}(A \oplus B)$.  Then $a$ and $d$ are invertible.
\end{cor}
\subsection*{Localization of an Additive Category}
Let $\mathcal{A}$ be an additive category, and let $\mathcal{B} \subset \mathcal{A}$ be a full additive subcategory closed under summands.  Let $e:\mathcal{B} \rightarrow \mathcal{A}$ be the inclusion and $s:\mathcal{A} \rightarrow \mathcal{A}/\mathcal{B}$ the quotient of $\mathcal{A}$ by the ideal consisting of morphisms factoring through $\mathcal{B}$.  Let $w \subset \mathcal{A}$ be the multiplicative set consisting of maps which are a composition of the form
\[
\xymatrix{
A \ar[r]^-i	& A \oplus B \ar[r]^{\cong}	& A' \oplus B' \ar[r]^-p	& A'
}
\]
with $A$ and $A'$ in $\mathcal{A}$, $B$ and $B'$ in $\mathcal{B}$, $i$ and $p$ the canonical inclusion and projection, and the middle map an isomorphism.  It is easy to check that $w$ is closed under sums and compositions.  An additive functor out of $\mathcal{A}$ sends all morphisms in $w$ to isomorphisms iff it sends all objects in $\mathcal{B}$ to zero, so $s$ is initial among additive functors inverting $w$.  In this sense $\mathcal{A}/\mathcal{B}$ is simultaneously the quotient of $\mathcal{A}$ by $\mathcal{B}$ and the localization of $\mathcal{A}$ at $w$.

Consider the following conditions on a morphism $f:C \rightarrow C'$:
\begin{itemize}
\item[1.]
For any choice of isomorphisms $\alpha:C \xrightarrow{\sim} A \oplus B$, $\alpha ':C' \xrightarrow{\sim} A' \oplus B'$, with $B,B'$ in $\mathcal{B}$ and $A,A'$ not having any nonzero summands in $\mathcal{B}$, the composition
\begin{equation} \label{composition}\xymatrix{
A \ar[r]^-{i_A}	& A \oplus B \ar[r]^{\alpha 'f\alpha ^{-1}}	& A' \oplus B' \ar[r]^-{p_A'}	& A'
} \end{equation}
is an isomorphism.
\item[2.]
For some choice of isomorphisms $\alpha:C \xrightarrow{\sim} A \oplus B$, $\alpha ':C' \xrightarrow{\sim} A' \oplus B'$, with $B,B'$ in $\mathcal{B}$ and $A,A'$ not having any nonzero summands in $\mathcal{B}$, (\ref{composition}) is an isomorphism.
\item[3.]
$f$ is in $w$.
\end{itemize}
\begin{lem}
\label{weakequivalences}
Among the above conditions, $1 \Rightarrow 2 \Rightarrow 3$.  If $\mathcal{A}$ is Krull-Schmidt, then $3 \Rightarrow 1$.
\end{lem}
\begin{proof}\ 
\begin{itemize}
\item[$1 \Rightarrow 2$] Trivial.
\item[$2 \Rightarrow 3$] Say $\alpha 'f\alpha ^{-1}:A \oplus B \rightarrow A' \oplus B'$ is given by the matrix $\left(\begin{smallmatrix} \phi & a \\ b & c \end{smallmatrix}\right)$.  Then
\[ \xymatrix{
A \oplus B \ar[r]^{\alpha 'f\alpha ^{-1}} \ar[d]_{i_{A \oplus B}}								& A' \oplus B' \\
A \oplus B \oplus B' \ar[r]_{\left(\begin{smallmatrix}\phi & a & 0 \\ 0 & 1 & 0 \\ b & c & 1 \end{smallmatrix}\right)}	& A' \oplus B \oplus B' \ar[u]^{p_{A' \oplus B'}}
} \]
commutes, and the lower horizontal map is an isomorphism with inverse $\left(\begin{smallmatrix}\phi ^{-1} & -\phi ^{-1} a & 0 \\ 0 & 1 & 0 \\ -b \phi ^{-1} & \, b \phi ^{-1} a - c\, & 1 \end{smallmatrix}\right)$.  So $\alpha 'f\alpha ^{-1}$ is in $w$, and therefore $f$ is in $w$.
\item[$3 \Rightarrow 1$]
Assume $\mathcal{A}$ is Krull-Schmidt and $f$ is in $w$.  Given decompositions $\alpha:C \xrightarrow{\sim} A \oplus B$, $\alpha ':C' \xrightarrow{\sim} A' \oplus B'$, there is by hypothesis a commuting diagram as below.
\[ \xymatrix{
A \ar[r] \ar[d]_{i_A}						& A' \\
A \oplus B \ar[r]^{\alpha 'f\alpha ^{-1}} \ar[d]_{i_{A \oplus B}}	& A' \oplus B' \ar[u]^{p_{A'}} \\
A \oplus B \oplus B_1 \ar[r]^{\cong}				& A' \oplus B' \oplus B_1' \ar[u]^{p_{A' \oplus B'}}
} \]
By Corollary \ref{sec:2}, the top horizontal map is an isomorphism.
\end{itemize}
\end{proof}
\begin{rem}
If $\mathcal{A}$ is Krull-Schmidt, it follows from the above characterization of maps in $w$ that $w$ satisfies the 2-out-of-3 property: if two out of $f$, $g$, and $f \circ g$ are in $w$ then all three are.  In this case $w$ is exactly the class of maps in $\mathcal{A}$ which become invertible in $\mathcal{A}/\mathcal{B}$.
\end{rem}
\subsection*{The Exact Sequence}
Now we proceed to our destination, Theorem \ref{K1}.  We adopt the notation of the previous section, with the added assumption that $\mathcal{A}$ is Krull-Schmidt.  In this case $\mathcal{B}$ and $\mathcal{A}/\mathcal{B}$ are automatically Krull-Schmidt.

\begin{lem}
\label{liftunits}
Suppose $A$ and $A'$ are objects of $\mathcal{A}$ with no nonzero summands in $\mathcal{B}$.  Suppose $\psi \in \mathcal{A}(A,A')$, and assume the image $\overline{\psi}$ of $\psi$ in $\mathcal{A}/\mathcal{B}$ is an isomorphism.  Then $\psi$ is itself an isomorphism.
\end{lem}

\begin{proof}
First we show that \[
\mathrm{ker}(\mathrm{End}_{\mathcal{A}}A \rightarrow \mathrm{End}_{\mathcal{A}/\mathcal{B}}A) \subset \mathrm{rad}(\mathrm{End}_{\mathcal{A}}A).
\]
That is, we show that given morphisms $\xymatrix{A \ar[r]^f & B \ar[r]^g & A}$ with $B$ in $\mathcal{B}$, $gf \in \mathrm{rad}(\mathrm{End}_{\mathcal{A}}A)$.  It suffices, using the characterization of $\mathrm{rad}(\mathrm{End}_{\mathcal{A}}A)$ in the proof of Lemma \ref{invertibility-criterion}, to show that for each indecomposable summand $A_i \subset A$, the induced map $\xymatrix{A_i \ar[r]^{f_i} & B \ar[r]^{g_i} & A_i}$ is in $\mathrm{rad}(\mathrm{End}_{\mathcal{A}}A_i)$.  But this follows from the fact that $\mathrm{End}_{\mathcal{A}}A_i$ is local and $A_i$ is not a summand of $B$.

Since $\overline{\psi}$ is invertible, there is $\phi \in \mathcal{A}(A',A)$ such that $\phi \psi - \mathrm{id}_A \in \mathrm{ker}(\mathrm{End}_{\mathcal{A}}A \rightarrow \mathrm{End}_{\mathcal{A}/\mathcal{B}}A)$ and $\psi \phi - \mathrm{id}_{A'} \in \mathrm{ker}(\mathrm{End}_{\mathcal{A}}A' \rightarrow \mathrm{End}_{\mathcal{A}/\mathcal{B}}A')$.  Therefore $\phi \psi - \mathrm{id}_A \in \mathrm{rad}(\mathrm{End}_{\mathcal{A}}A)$ and (by the same argument) $\psi \phi - \mathrm{id}_{A'} \in \mathrm{rad}(\mathrm{End}_{\mathcal{A}}A')$, so $\phi \psi$ and $\psi \phi$ are both invertible.  It follows that $\psi$ is invertible.
\end{proof}

\begin{lem}
\label{elimination}
Suppose $A$ is an object of $\mathcal{A}$ with no nonzero summand in $\mathcal{B}$, $B$ is an object of $\mathcal{B}$, and $\alpha = \left( \begin{smallmatrix}\phi & a \\ b & c \end{smallmatrix} \right) \in \mathrm{Aut}_{\mathcal{A}}(A \oplus B)$.  Then $[\alpha ]-[\phi ] \in \mathrm{im}\,K_1(e).$
\end{lem}
\begin{proof}
Note that $\phi$ is an automorphism by Corollary \ref{sec:2}.  Using Remark \ref{reduction}, compute:
\begin{align*}
\left[ A \oplus B, \left(\begin{smallmatrix}\phi & a \\ b & c \end{smallmatrix}\right)\right] &= \left[A \oplus B, \left(\begin{smallmatrix}\phi & a \\ 0 & c-b\phi ^{-1}a\end{smallmatrix}\right)\right] \\
&= \left[A \oplus B, \left(\begin{smallmatrix}\phi & 0 \\ 0 & c-b\phi ^{-1}a\end{smallmatrix}\right)\right] \\
&= [A,\phi ] + [B, c-b\phi ^{-1}a] \\
&\equiv [A,\phi ]\ (\mathrm{mod\ im}\,K_1(e))
\end{align*}
\end{proof}

\begin{lem}
\label{equality}
Suppose $(A,\phi )$ and $(A,\psi )$ are objects of $\mathrm{Aut}(\mathcal{A})$ such that $\overline{\phi} = \overline{\psi}$ in $\mathcal{A}/\mathcal{B}$.  Then $[\phi] - [\psi] \in \mathrm{im}\,K_1(e)$.
\end{lem}

\begin{proof}
Since $\overline{\phi} = \overline{\psi}$, $\phi - \psi$ factors as a composition $\xymatrix{A \ar[r]^f & B \ar[r]^g & A}$ through some object $B$ of $\mathcal{B}$.  Using Remark \ref{reduction}, compute:
\begin{align*}
[A,\phi] &= \left[A \oplus B, \left(\begin{smallmatrix}\phi & 0 \\ 0 & 1_B\end{smallmatrix}\right)\right] \\
&= \left[A \oplus B, \left(\begin{smallmatrix}\phi & 0 \\ f & 1_B\end{smallmatrix}\right)\right] \\
&= \left[A \oplus B, \left(\begin{smallmatrix}\phi - gf & -g \\ f & 1_B\end{smallmatrix}\right)\right] \\
&\equiv [A,\psi]\ (\mathrm{mod}\ \mathrm{im}\,K_1(e)),\mbox{ by Lemma \ref{elimination}}
\end{align*}
Therefore $[A,\phi] \equiv [A,\psi]\ (\mathrm{mod}\ \mathrm{im}\,K_1(e))$.
\end{proof}

\begin{lem}
\label{isomorphism}
Suppose $(C,\alpha )$ and $(C',\alpha ')$ are objects of $\mathrm{Aut}(\mathcal{A})$ such that $(C,\overline{\alpha }) \cong (C',\overline{\alpha '})$ in $\mathrm{Aut}(\mathcal{A}/\mathcal{B})$.  Then $[\alpha]-[\alpha '] \in \mathrm{im}\,K_1(e)$.
\end{lem}

\begin{proof}
We may assume there are decompositions
\begin{align*}(C,\alpha) = \left( A \oplus B,\left( \begin{smallmatrix} \phi \phantom{'} & a \\ b \phantom{'} & c \end{smallmatrix} \right) \right) & & (C',\alpha ') = \left( A' \oplus B',\left( \begin{smallmatrix}\phi '& a' \\ b' & c' \end{smallmatrix} \right) \right)\end{align*}
such that $A$ and $A'$ have no nonzero summand in $\mathcal{B}$ and $B$ and $B'$ are in $\mathcal{B}$.  By assumption there is $\beta :A \oplus B \rightarrow A' \oplus B'$ such that $\overline{\beta}$ is an isomorphism in $\mathcal{A}/\mathcal{B}$ and $\overline{\beta} \overline{\alpha} = \overline{\alpha '} \overline{\beta}$.  Say $\beta$ is given by a matrix of the form $\left( \begin{smallmatrix}\psi & * \\ * & * \end{smallmatrix} \right)$.  By Lemma \ref{liftunits}, $\psi$ must be an isomorphism.  Now, $\overline{\beta} \overline{\alpha} = \overline{\alpha '} \overline{\beta} \Rightarrow \overline{\psi}\overline{\phi}=\overline{\phi '}\overline{\psi} \Rightarrow \overline{\phi}=\overline{\psi ^{-1} \phi ' \psi}$ so
\begin{align*}
[C,\alpha] &\equiv [A, \phi ]\ (\mathrm{mod}\ \mathrm{im}\,K_1(e)),\mbox{ by Lemma \ref{elimination}}\\
&\equiv [A, \psi ^{-1}\phi ' \psi ]\ (\mathrm{mod}\ \mathrm{im}\,K_1(e)),\mbox{ by Lemma \ref{equality}}\\
&= [A', \phi '] \\
&\equiv [C',\alpha ']\ (\mathrm{mod}\ \mathrm{im}\,K_1(e)),\mbox{ by Lemma \ref{elimination}}
\end{align*}
\end{proof}
\begin{thm}
\label{K1}
Suppose $\mathcal{A}$ is Krull-Schmidt.  Then the sequence
\[ \xymatrix{
K_1(\mathcal{B}) \ar[r]^{K_1(e)}	& K_1(\mathcal{A}) \ar[r]^{K_1(s)}	& K_1(\mathcal{A}/\mathcal{B}) \ar `r[d] `[l] `[lld]_{0} `[dl] [dl] \\
				& K_0(\mathcal{B}) \ar[r]^{K_0(e)}	& K_0(\mathcal{A}) \ar[r]^{K_0(s)}	& K_0(\mathcal{A}/\mathcal{B}) \ar[r]	& 0
} \]
is exact.
\end{thm}

\begin{proof}
We show only that the sequence is exact at $K_1(\mathcal{A}/\mathcal{B})$ and $K_1(\mathcal{A})$; the rest is easy.
To show $K_1(s)$ is surjective, take $[A,\phi] \in K_1(\mathcal{A}/\mathcal{B})$; after replacing $(A,\phi)$ by an isomorphic object of $\mathrm{Aut}(\mathcal{A}/\mathcal{B})$, we may assume $A$ has no nonzero summands in $\mathcal{B}$.  Using Lemma \ref{liftunits}, we see that any lift $\widetilde{\phi}$ of $\phi$ to $\mathcal{A}$ is automatically an automorphism of $A$.  Hence $K_1(s)([\widetilde{\phi}]) = [\phi]$, as desired.

Next we show the sequence is exact at $K_1(\mathcal{A})$.  Since $K_1(s)$ is surjective, any element of $\mathrm{ker}\,K_1(s)$ may be written as a sum of elements of the form
\begin{itemize}
\item[1.]
$[A,\phi]-[A',\phi ']$, for some $\phi \in \mathrm{Aut}_{\mathcal{A}}A$ and $\phi' \in \mathrm{Aut}_{\mathcal{A}}A'$ with $(A,\overline{\phi }) \cong (A',\overline{\phi '})$ in $\mathrm{Aut}(\mathcal{A}/\mathcal{B})$
\item[2.]
$[A,\phi ]-[A,\alpha ]-[A,\beta ]$ for some $\phi , \alpha , \beta \in \mathrm{Aut}_{\mathcal{A}}A$ with $\overline {\alpha \beta} = \overline{\phi}$ in $\mathcal{A}/\mathcal{B}$
\item[3.]
$[A,\phi ]-[A',\phi ']-[A'',\phi '']$, for some $(A,\phi )$, $(A',\phi ')$, and $(A'',\phi '')$ in $\mathrm{Aut}(\mathcal{A})$ such that there is a conflation in $\mathrm{Aut}(\mathcal{A}/\mathcal{B})$
\[ \xymatrix{
(A',\overline{\phi '})\ \ar@{>->}[r]	& (A,\overline{\phi }) \ar@{>>}[r]	& (A'',\overline{\phi ''})
} \]
\end{itemize}
We need to check that any such element is in $\mathrm{im}\,K_1(e)$.  Since an element of the first form is also of the third form, we skip the check for elements of the first form.  For elements of the second form, observe that
\[
[A,\phi ] - [A,\alpha ] - [A,\beta ] = [A,\phi ] - [A,\alpha \beta ] \in \mathrm{im}\,K_1(e)
\] by Lemma \ref{equality}.  So it remains to show any element of the third form is in $\mathrm{im}\, K_1(e)$.

Given $(A,\phi )$, $(A',\phi ')$, and $(A'',\phi '')$ as in 3. above, there is an isomorphism in $\mathrm{Aut}(\mathcal{A}/\mathcal{B})$
\[
(A,\overline{\phi}) \cong \left( A'\oplus A'',\left( \begin{smallmatrix}\overline{\phi}' & \overline{h} \\ 0 & \overline{\phi}''\end{smallmatrix}\right) \right)
\]
for some $h:A'' \rightarrow A'$.  Since $\left( \begin{smallmatrix}\phi ' & h \\ 0 & \phi ''\end{smallmatrix}\right)$ is invertible, it follows by Lemma \ref{isomorphism} that
\begin{align*}
[A,\phi ] &\equiv \left[A' \oplus A'', \left( \begin{smallmatrix}\phi ' & h \\ 0 & \phi ''\end{smallmatrix}\right) \right]\ (\mathrm{mod}\ \mathrm{im}\,K_1(e)) \\
&= [A',\phi ']+[A'',\phi '']
\end{align*}
and therefore $[A,\phi]-[A',\phi ']-[A'',\phi ''] \in \mathrm{im}\,K_1(e)$.
\end{proof}

\begin{rem}
There is a one-to-one correspondence between equivalence classes of Krull-Schmidt categories with finitely many indecomposables and Morita classes of semiperfect rings.  The correspondence is defined by assigning to a category $\mathcal{A}$ the ring $(\mathrm{End}_{\mathcal{A}}(\bigoplus \limits_{[M] \in \mathrm{ind}(\mathcal{A})} M))^{\mathrm{op}}$, and by assigning to a ring $R$ the category $\mathsf{proj(R)}$.  Using this correspondence we may restate Theorem \ref{K1} as follows.
Let $R$ be a semiperfect ring and $e \in R$ an idempotent.  Let
\[
S = \{x \in R\ |\ (1-e)x(1-e) \in R^{\times}\}.
\]
Then $R/ReR$ coincides with the localization $S^{-1}R$ of $R$ at $S$; that is, the projection $p:R \rightarrow R/ReR$ is initial among $S$-inverting ring homomorphisms.  Let $f:R \rightarrow eRe$ be the ring homomorphism $x \mapsto exe$.  Then the following sequence is exact.
\[ \xymatrix{
K_1(eRe) \ar[r]^{f^*}	& K_1(R) \ar[r]^{p_*}	& K_1(S^{-1}R) \ar `r[d] `[l] `[lld]_{0} `[dl] [dl]	& 			& \\
			& K_0(eRe) \ar[r]^{f^*}	& K_0(R) \ar[r]^{p_*}					& K_0(S^{-1}R) \ar[r]	& 0
} \]
\end{rem}

\begin{rem}
Theorem \ref{K1} does not seem to follow from known localization theorems in $K$-theory.  In particular, since the functor $\mathsf{D}^b(s):\mathsf{D}^b(\mathcal{A}) \rightarrow \mathsf{D}^b(\mathcal{A}/\mathcal{B})$ between bounded derived categories may not be full, it may not induce an equivalence between the Verdier quotient $\mathsf{D}^b(\mathcal{A})/\mathsf{D}^b(\mathcal{B})$ (or even its idempotent completion) and $\mathsf{D}^b(\mathcal{A}/\mathcal{B})$.  And Theorem \ref{K1} does not follow from \cite[Theorem 0.5]{NeeRan04}: the $\mathrm{Tor}$-condition in that theorem does not seem to be satisfied in this situation.
\end{rem}

\bibliographystyle{amsalpha}
\bibliography{ref}

\end{document}